\documentclass{amsart}

\usepackage[all]{xy}   % xypic package
\usepackage{hyperref} %to generate pdf with links

\usepackage[usenames,dvipsnames]{xcolor}
\theoremstyle{plain}
\newtheorem{lem}{Lemma}[section]

\newtheorem{prop}[lem]{Proposition}
\newtheorem{thm}[lem]{Theorem}

\theoremstyle{definition}
\newtheorem{ex}[lem]{Example}
\newtheorem{rem}[lem]{Remark}
\newtheorem{dfn}[lem]{Definition}

\newcommand{\Z}{\mathbb{Z}}               % integers
\newcommand{\LL}{\mathbb{L}}                % Lazard ring

\newcommand{\DF}{\mathbb{D}_F}   % formal affine Demazure algebra
\newcommand{\DFd}{\DF^\star}          % dual of the algebra

\newcommand{\Ll}{\mathcal{L}}   % line bundle

\DeclareMathOperator{\CH}{\mathrm{CH}}      % Chow ring
\DeclareMathOperator{\Hom}{\mathrm{Hom}}      % Hom space
\DeclareMathOperator{\hh}{\mathtt{h}}             % oriented cohomology

\newcommand{\de}{\delta}                   
\newcommand{\al}{\alpha}
\newcommand{\be}{\beta}
\newcommand{\la}{\lambda}

\begin{document}

\title[Localized operations]{Localized operations on $T$-equivariant oriented cohomology of projective homogeneous varieties}

\author[K.~Zainoulline]{Kirill Zainoulline}
\address[Kirill Zainoulline]{Department of Mathematics and Statistics, University of Ottawa, 150 Louis-Pasteur, Ottawa, ON, K1N 6N5, Canada}
\email{kirill@uottawa.ca}
\urladdr{http://mysite.science.uottawa.ca/kzaynull/}


\begin{abstract} 
In the present paper we provide a general algorithm to compute multiplicative cohomological operations 
on algebraic oriented cohomology of projective homogeneous $G$-varieties, 
where $G$ is a split reductive algebraic group over a field $k$ of characteristic $0$.
More precisely, we extend such operations to the respective
$T$-equivariant ($T$ is a maximal split torus of $G$) oriented theories, and then compute them using equivariant Schubert calculus techniques.
This generalizes an approach suggested by Garibaldi-Petrov-Semenov for Steenrod operations.
We also show that operations on the theories of additive type commute with classical push-pull operators up to a twist.
\end{abstract}

\maketitle

%%%%%%%%%%%%%%%%%%%%%%%%%%%%%
%%%%%%%%%%%%%%%%%%%%%%%%%%%%%
%%%%%%%%%%%%%%%%%%%%%%%%%%%%%

\section{Introduction}

Cohomological operations can be viewed as morphisms between algebraic oriented cohomology theories (AOCTs). A classical example of such operation is the Chern character which gives a morphism between the $K$-theory and the Chow theory with rational coefficients. 
Another celebrated examples are the Steenrod (Landweber-Novikov) operations which can be viewed as endomorphisms of the Chow theory with finite coefficients (resp. of cobordism). Observe that following a general Riemann-Roch formalism to each cohomological operation one can assign the respective Riemann-Roch type formula which essentially describes its behaviour with respect to push-forwards (see e.g. Panin~\cite{Pa04}). In this way the Chern character leads to the classical Grothendieck-Riemann-Roch theorem \cite{SGA6}, the Steenrod operations for closed embeddings give the Wu formula and for projections -- various versions of the Rost degree formulas (see Levine~\cite{Le07} and Merkurjev~\cite{Me03}).

When dealing with operations one of the major problems is to effectively compute them for some given class of varieties. For example, the problem of computing the Steenrod operations on Chow groups of twisted flag varieties has been a subject of intensive investigations until recently. Motivated by the use of Steenrod operations in the proof of the Milnor conjecture by Voevodsky, it became an important tool (see Brosnan~\cite{Br03} and Merkurjev~\cite{Me03}) in the study of algebraic cycles and motives of quadrics and Severi-Brauer varieties. In the recent paper by Garibaldi-Petrov-Semenov~\cite{GPS} the Steenrod operations have been used to describe motivic decomposition types of twisted flag varieties. As for Landweber-Novikov operations, we refer to pioneer works by Vishik (see e.g. \cite{Vi17}) where these operations have been successfully applied to study algebraic cobordism.

In the present paper we provide a general algorithm to compute such operations on projective homogeneous $G$-varieties, where $G$ is a split reductive algebraic group over a field $k$ of characteristic $0$. Its main point is 
\begin{itemize}
\item[(i)] to lift a given operation to the `$T$-equivariant cohomology level', where $T$ is a maximal split torus of $G$, and then 
\item[(ii)] to apply equivariant Schubert calculus techniques to compute it.
\end{itemize}
Observe that in the case of Chow groups and Steenrod operations such algorithm has been already proposed in \cite{GPS} which can be viewed as a starting point of the present work.

Our key tools are the Riemann-Roch formalism for AOCTs by Panin~\cite{Pa04} and Levine-Morel~\cite{LM}, Vishik's results~\cite{Vi17} on the correspondence between (multiplicative) operations and morphisms of formal group laws, equivariant localization techniques of Kostant-Kumar~\cite{KK86,KK90} and their generalizations to arbitrary AOCTs~\cite{HMSZ,CZZ,CZZ1,CZZ2}. 

 As for (i) note that equivariant extensions of the Chern character have been studied before in the context of the Grothendieck-Riemann-Roch theorem by many authors (see e.g.~\cite{EG00} and recently~\cite{AGP}). Our main result (Theorem~\ref{thm:main}) states that any operation on AOCTs (e.g. the Chern character or the Steenrod operation) can be extended to the $T$-equivariant oriented cohomology in a way compatible with forgetful and characteristic maps. 
 
 As for (ii) we describe explicit formulas and steps needed to compute the operations on both the Schubert and the dual bases (Section~\ref{sec:compop}). We also show that multiplicative operations on theories of additive type commute with classical push-pull operators up to a twist (Proposition~\ref{thm:Riem}).
  
The paper is organized as follows. In section~\ref{sec:fgl} we recall basic properties of formal group laws, algebraic oriented cohomology theories and multiplicative operations. In the next section~\ref{sec:twth} we provide examples of multiplicative operations (e.g. Landweber-Novikov and Steenrod operations) using the techniques of twisted theories. In section~\ref{sec:main} we discuss $T$-equivariant oriented cohomology theories, forgetful and characteristic maps and we state and prove our main Theorem~\ref{thm:main}. In section~\ref{sec:compop} we describe the algorithms used to compute cohomological operations. In the last section, we study operations on theories of additive type and its behaviour with respect to push-pull operators.

\paragraph{\it Acknowledgements}  Partially supported by the NSERC Discovery grant RGPIN-2015-04469, Canada. 

%%%%%%%%%%%%%%%%%%%%%%%%%%%%%
%%%%%%%%%%%%%%%%%%%%%%%%%%%%%
%%%%%%%%%%%%%%%%%%%%%%%%%%%%%

\section{Formal group laws, oriented cohomology and operations}\label{sec:fgl}

\subsection{\it Formal group laws} We recall definition and basic properties of one-dimensional commutative formal group laws following \cite{La55}.

Let $R$ be a commutative unital ring. By a formal group law $F$ over $R$ (FGL $F/R$) we call a power series in two variables $F(x,y)\in R[[x,y]]$ also denoted $x+_F y$ which is commutative $x+_F y=y+_F x$, associative $(x+_F y)+_F z=x+_F(y+_F z)$ and has neutral element $x+_F 0=x$. By definition, we have
\[
x+_F y=x+y+\sum_{i,j>0} a_{ij}x^i y^j,\;\text{ for some coefficients }a_{ij}\in R.
\]
The ring $R$ is called the coefficient ring of $F$.

The quotient of the polynomial ring in $a_{ij}$s modulo relations imposed by the commutativity and associativity gives the Lazard ring $\LL$. The universal FGL over $\LL$ is denoted $F_U$. For any FGL $F$ over $R$ there is the map $\LL \to R$ given by evaluating the coefficients $a_{ij}$ of $F_U$.

By a morphism $F_1/R_1 \to F_2/R_2$ between FGLs we call a pair $(\phi,\gamma)$, where $\phi\colon R_1\to R_2$ is a ring homomorphism and $\gamma\in R_2[[x]]$ such that
\[
\phi(F_1)(\gamma(x),\gamma(y))=\gamma(F_2(x,y)).
\]

\begin{ex}
Basic examples of FGLs are 

(i) The additive FGL $x+_F y=x+y$ over $R=\Z$ with $F_U\to F$ given by $a_{ij}\mapsto 0$;

(ii) The multiplicative FGL $x+_F y=x+y-\be xy$ over $R=\Z[\be]$ with $F_U\to F$ given by $a_{11}\mapsto -\be$, $a_{ij}\mapsto 0$ for all $i+j>2$. If $\be$ is invertible and $R=\Z[\be^{\pm 1}]$, then the respective $F$ is called multiplicative periodic.
\end{ex}

\subsection{\it Algebraic oriented cohomology}
Let $\hh$ be a graded algebraic oriented cohomology theory over a field $k$ of characteristic $0$ introduced in~\cite[\S1]{LM}, i.e.~$\hh$ is a contravariant functor from the category $\mathbf{Sm}_k$ of smooth quasi-projective varieties over $k$ to the category of graded commutative rings which satisfies the axioms of \cite[Definition~1.1.2]{LM} (see also~\cite[Definition~2.1]{Vi17}). We assume in addition that $\hh$ satisfies the localization axiom \cite[\S2.(EXCI)]{Vi17} and it is generically constant in the sense of \cite[\S2.(CONST)]{Vi17}.

The word `oriented' here reflects the fact that there is a 1-1 correspondence between the Euler structures (characteristic classes) on $\hh$ (see \cite{Pa04}):
\begin{quote}
To each line bundle $\Ll$ over a variety $X$ one assigns the characteristic (Euler) class $c_1^{\hh}(\Ll):=s^*s_*(1_X) \in \hh^1(X)$, where $s$ is the zero-section of $\Ll$ and $s^*$ (resp. $s_*$) is the induced pull-back (resp. push-forward); 
\end{quote} 
and the choices of a local parameter $x^{\hh}$ (orientations):
\begin{quote}
Choose $x^{\hh}=c_1^{\hh}(\mathcal{O}(1)) \in \hh^1(\mathbb{P}^\infty)=R[[x]]^{(1)}$, where $R=\hh(pt)$, $pt=\mathrm{Spec}\, k$ and $\deg x=1$. Observe that $x^{\hh}=c_0x+c_1x^2+\ldots$, where $c_0$ is invertible in $R$ and $\deg c_i=-i$.
\end{quote}

Any such theory satisfies the Quillen formula \cite[Lemma~1.1.3]{LM}. Namely, given two line bundles $\Ll_1$, $\Ll_2$ over $X$ one has
\[
c_1^{\hh}(\Ll_1\otimes \Ll_2)=c_1^{\hh}(\Ll_1)+_F c_1^{\hh}(\Ll_2),
\]
where $F$ is the FGL~over the coefficient ring $R$. To respect the grading on $\hh$ we set $\deg a_{ij}=1-i-j$ for the coefficients $a_{ij}$ of $F$.

\begin{ex}
Examples of AOCTs are

(i) The Chow theory $\hh(X)=\CH(X;\Z)$ (Chow groups modulo rational equivalence) which corresponds to the additive FGL over $R=\Z$ \cite[Ex.1.1.4]{LM}; 

(ii) The graded $K$-theory $\hh(X)=K_0(X)[\be^{\pm 1}]$, where $\be$ is a formal variable (here we assign $\deg \be=-1$ and $\deg x=0$ for all $x\in K_0(X)$), which corresponds to the multiplicative periodic FGL over $R=\Z[\be^{\pm 1}]$ \cite[Ex.1.1.5]{LM}. 

(iii) The algebraic cobordism $\hh(X)=\Omega(X)$ of Levine-Morel~\cite{LM} which corresponds to the universal FGL over $R=\LL$.
\end{ex}

In the opposite direction, given a FGL~$F/R$ one defines an AOCT called a free theory by changing the coefficients 
\[
\hh(-)=\Omega(-)\otimes_{\LL} R.
\]
Observe that both the Chow theory and the graded K-theory are free theories.

\subsection{\it Multiplicative operations} Following \cite{Vi17} we now introduce multiplicative operations as follows.

Let $\hh_1$ and $\hh_2$ be AOCTs. According to \cite[Definition~3.6]{Vi17}, a multiplicative operation $\mathcal{C}\colon \hh_1 \to \hh_2$ is a natural transformation between $\hh_1$ and $\hh_2$ considered as contravariant functors from ${\bf Sm}_k$ to the category of rings (not necessarily graded).

We will extensively use \cite[Theorem~6.9]{Vi17} which says that 
\begin{quote}
If $\hh_1$ is a free theory, then there is a 1-1 correspondence between multiplicative operations $\hh_1\to \hh_2$ and morphisms $F_1/R_1\to F_2/R_2$ of FGLs.
\end{quote}

Namely, let $F_1/R_1$ and $F_2/R_2$ be the FGLs corresponding to $\hh_1$ and $\hh_2$ respectively. To a multiplicative operation $\mathcal{C}$ one assigns the morphism of FGLs
\[
(\mathcal{C}(pt),\gamma_{\mathcal{C}})\colon F_1/R_1\to F_2/R_2,
\]  
where $\mathcal{C}(pt)\colon R_1\to R_2$ is the induced map on the coefficient rings and $\gamma_{\mathcal{C}} \in R_2[[x]]$ is the power series which expresses $\mathcal{C}(x^{\hh_1})$ in terms of $x^{\hh_2}$. The series $\gamma_{\mathcal{C}}(x)$ defines the inverse Todd genus $\mathrm{itd}(x)=\gamma_{\mathcal{C}}(x)/x$ (see e.g. \cite{Pa04}). Observe that the composition of multiplicative operations corresponds to the composition of morphisms of FGLs. 

\begin{ex}
Set $\hh_1(-)=K_0(-)[\beta]$ to be the graded $K$-theory and $\hh_2(-)=\CH(-;\mathbb{Q})$ the Chow theory with rational coefficients. The Chern character is the multiplicative operation $\mathrm{ch} \colon \hh_1 \to \hh_2$ with $\gamma_{\mathrm{ch}}(x)=1-e^{-x}$.
\end{ex}

In the opposite direction (following \cite{Vi17}), given a morphism $(\phi,\gamma)\colon F_1/R_1 \to F_2/R_2$ of FGLs, one defines homomorphisms for all $r\ge 0$
\begin{equation}\label{eq:vfgl}
\mathcal{C}\colon R_1[[x_1^{\hh_1},\ldots,x_r^{\hh_1}]] \to R_2[[x_1^{\hh_2},\ldots,x_r^{\hh_2}]]
\end{equation}
by the rule
\[
\mathcal{C}(f(x_1^{\hh_1},\ldots,x_r^{\hh_1})):=\phi(f)(\gamma(x_1^{\hh_2}),\ldots,\gamma(x_r^{\hh_2})).
\]
Identifying $R_i[[x_1^{\hh_i},\ldots,x_r^{\hh_i}]]$ with  $\hh_i((\mathbb{P}^\infty)^{\times r})$ one obtains an operation $\mathcal{C}$ on products of projective spaces which one then extends uniquely to an operation $\mathcal{C}$ on ${\bf Sm}_k$ by \cite[Theorem~5.1]{Vi17}.

In the present work we exploit the same idea: 
Replacing $\hh((\mathbb{P}^\infty)^{\times r})$ by a $T$-equivariant AOCT of a point (here $T$ stands for a split torus of rank $r$) we extend $\mathcal{C}$ to a natural transformation $\mathcal{C}_T$ of $T$-equivariant AOCTs which we call a localized multiplicative operation.

%%%%%%%%%%%%%%%%%%%%%%%%%%%%%
%%%%%%%%%%%%%%%%%%%%%%%%%%%%%
%%%%%%%%%%%%%%%%%%%%%%%%%%%%%

\section{Twisted cohomology and multiplicative operations} \label{sec:twth}

\subsection{\it Twisted cohomology}
We recall the notion of a twisted theory $\tilde\hh(X)$ which goes back to~\cite{Qu71}. 
In the presentation we follow \cite[\S4.1.9]{LM} and \cite[\S4]{Me02}.

Let $\hh$ be a (graded) AOCT with coefficient ring $R$. Set $R[\mathfrak{t}]=R[t_1,t_2,\ldots]$ to be the graded polynomial ring in variables $t_i$ with $\deg t_i=-i$ (here we denote by $\mathfrak{t}$ the set of all variables). Consider a generic power series
\[
\lambda_{\mathfrak{t}}(x)=x+\sum_{i\ge 1} t_i x^{i+1}\in R[\mathfrak{t}][[x]].
\]  
It has a formal inverse $\lambda^{-1}_{\mathfrak{t}}(x)=x-t_1x^2+(2t_1^2-t_2)x^3+o(x^3)$ defined by
\[
\lambda_{\mathfrak{t}}(\lambda^{-1}_{\mathfrak{t}}(x))=x=\lambda^{-1}_{\mathfrak{t}}(\lambda_{\mathfrak{t}}(x)).
\]
We define a theory $\tilde\hh$ twisted by means of $\lambda_{\mathfrak{t}}$ as follows:

First, we extend coefficients over $R[\mathfrak{t}]$ and obtain an AOCT \[\hh[\mathfrak{t}](X)=\hh(X)\otimes_R R[\mathfrak{t}]\] with the same formal group law 
but considered over $R[\mathfrak{t}]$ which we denote by $F[\mathfrak{t}]$. 

Then, we set $\tilde \hh(X)$ to be the same graded ring as $\hh[\mathfrak{t}](X)$ but we redefine the characteristic classes as
\[
c_1^{\tilde\hh}(\mathcal{L}):=\lambda_{\mathfrak{t}}(c_1^{\hh}(\mathcal{L}))
\]
and the formal group law as
\[
\tilde F(x,y):=\lambda_{\mathfrak{t}}(F[\mathfrak{t}](\lambda^{-1}_{\mathfrak{t}}(x),\lambda^{-1}_{\mathfrak{t}}(y)    )).
\]
By definition, $\lambda_{\mathfrak{t}}\colon F[\mathfrak{t}] \to \tilde F$ is a morphism of FGLs over $R[\mathfrak{t}]$. 

\begin{ex}\label{ex:LN}
Observe that a matrix expressing the coefficients $\tilde a_{ij}\in R[\mathfrak{t}]$ of the twisted FGL
\[
\tilde F(x,y)=x+y+\sum_{i,j\ge 1} \tilde a_{ij}x^iy^j
\]
in terms of the coefficients $a_{i'j'}$ is a unitriangular matrix with coefficients $p_{i'j'} \in R[\mathfrak{t}]$ given by homogeneous polynomials:
\[
\tilde a_{ij} = a_{ij}+ \sum_{i'+j'<i+j}a_{i'j'} p_{i'j'},\quad \deg p_{i'j'}=i+j-i'-j'.
\]
In particular, for the first two coefficients we have
\[
\tilde a_{11}  =a_{11}+2t_1,\qquad  \tilde a_{12} =a_{12}+a_{11}t_1-2t_1^2.
\]
\end{ex}

\begin{rem} 
The theory $\tilde\hh$ has the same pull-backs as $\hh$ (but it has twisted push-forwards). Moreover, the twisted theory $\widetilde{\CH}$ can be viewed as an approximation for the algebraic cobordism $\Omega$ (the Lazard ring $\LL$ embeds into $\widetilde{\CH}(pt)$).
\end{rem}

\subsection{\it Landweber-Novikov operations} Consider the algebraic cobordism $\Omega$.
Following \cite[\S10]{Me02} we define the total Landweber-Novikov (LN) operation $\mathcal{C}^\Omega\colon \Omega \to \Omega[\mathfrak{t}]$ and the LN operations 
\[
\mathcal{C}_I^\Omega\colon \Omega^*(X) \to \Omega^{*+|I|}(X)
\] 
corresponding to partitions  $I=(i_1\ge i_2\ge \ldots \ge i_m>0)$, where $|I|=i_1+\ldots+i_m$ is the degree of $I$, as follows:

Consider the twisted theory $\tilde\Omega$. Since it is the AOCT, by universality of $\Omega$ there is the ring homomorphism $\Omega(X) \to \widetilde{\Omega}(X)$ which maps $c_1^\Omega(\Ll)$ to $c_1^{\widetilde{\Omega}}(\Ll)$. Identifying $\tilde\Omega(X)$ with $\Omega[\mathfrak{t}](X)$ as graded rings we obtain the total LN operation $\mathcal{C}^\Omega$ which is the graded multiplicative operation corresponding to the morphism $(\phi,\lambda_{\mathfrak{t}})\colon F_U \to F_U[\mathfrak{t}]$ of FGLs, where $\phi(a_{ij})=\tilde a_{ij}$ (cf. \cite[Example~3.9]{Vi17}).

Taking the coefficient at the monomial $\mathfrak{t}_I=t_{i_1}t_{i_2}\ldots t_{i_m}$ of the image of $\mathcal{C}^\Omega$ gives the group homomorphism $\mathcal{C}_I^\Omega\colon \Omega^*(X) \to \Omega^{*+|I|}(X)$. In other words, we have
\[
\mathcal{C}^\Omega(x)=\sum_I \mathcal{C}_I^\Omega(x) \mathfrak{t}_I.
\]

\begin{ex}\label{ex:LNP}
In particular, if $X=pt$, then $\mathcal{C}_I^\Omega(a_{ij})$ is the coefficient at $\mathfrak{t}_I$ in the expansion of $\tilde a_{ij}$ in terms of $a_{i'j'}$ of Example~\ref{ex:LN}.

Suppose $X=\mathbb{P}^n$. Then $\Omega(X)=\mathbb{L}[c]/(c^{n+1})\in \Omega^1(X)$, where $c=c_1^{\Omega}(\mathcal{O}(1)) \in \Omega^1(X)$. By definition $\mathcal{C}^\Omega\colon \Omega(X) \to \Omega[\mathfrak{t}](X)$ is the graded ring homomorphism which maps $a_{ij} \mapsto \tilde{a}_{ij}$ and $c\mapsto \lambda_{\mathfrak{t}}(c)=c+t_1c^2+t_2c^3+\ldots$.
\end{ex}

More generally, instead of taking the coefficient at $\mathfrak{t}_I$ we can apply an arbitrary graded $\LL$-linear map $\psi\colon \LL[\mathfrak{t}] \to \LL$ of degree $l$. This gives a group homomorphism 
\[
\mathcal{C}^\Omega_\psi \colon \Omega^*(X) \stackrel{\mathcal{C}^\Omega}\longrightarrow \Omega[\mathfrak{t}](X) \stackrel{-\otimes \psi}\longrightarrow \Omega^{*+l}(X)
\]
which is a $\LL$-linear combination of LN operations corresponding to partitions. 

\begin{rem}
Following \cite{Vi17} an additive operation is a natural transformation of functors $\hh_1^n \to \hh_2^m$ with values in the category of abelian group. The LN operations corresponding to partitions provide examples of additive operations. Moreover, by \cite[Theorem~3.18]{Vi17} all additive operations can be described as certain $\LL\otimes_{\Z} \mathbb{Q}$-linear combinations of LN operations corresponding to partitions.
\end{rem}

\subsection{\it Chow traces of LN operations} 
We follow \cite[\S3]{Me03}. Fix a prime $p$. Consider the Chow theory with coefficients in $R=\mathbb{F}_p$ denoted by $\mathrm{Ch}(-)$. Set $R[\mathfrak{t}']=R[t_{p-1},t_{p^2-1},\ldots]$, where $\mathfrak{t}'$ denote the subset of $\mathfrak{t}$ consisting of all variables indexed by integers of the form $p^r-1$, $r\ge 1$. Consider the theory $\widetilde{\mathrm{Ch}}$ twisted by means of  $\lambda_{\mathfrak{t}'}(x)=x+t_{p-1}x^p+t_{p^2-1}x^{p^2}+\ldots \in R[[x]]$. 

It turns out that the twisted FGL $\tilde{F}$ coincides with $F[\mathfrak{t}']$, hence it is also additive. By universality of $\mathrm{Ch}$ (among all theories with additive FGL) there is a ring homomorphism $\mathrm{Ch}(X) \to \widetilde{\mathrm{Ch}}(X)$. Identifying $\widetilde{\mathrm{Ch}}(X)$ with $\mathrm{Ch}[\mathfrak{t}'](X)$ as graded rings we obtain a graded multiplicative operation $\mathcal{C}^{\mathrm{Ch}}\colon \mathrm{Ch}\to \mathrm{Ch}[\mathfrak{t}']$ which we call the total Chow trace (of LN operation). Observe that $\mathcal{C}^{\mathrm{Ch}}$ corresponds to the morphism $(\mathrm{id},\lambda_{\mathfrak{t}'})\colon F\to F[\mathfrak{t}']$ of FGLs.

Let $I$ be a partition which consist only of integers of the form $p^r-1$, $r\ge 1$. Taking the coefficient at the monomial $\mathfrak{t}'_I$ we obtain a $\mathbb{F}_p$-linear map $\mathcal{C}^{\mathrm{Ch}}_I\colon \mathrm{Ch}(X) \to \mathrm{Ch}(X)$ which we call the Chow trace corresponding to $I$ and by definition, we have
\[
\mathcal{C}^{\mathrm{Ch}}(x)=\sum_I \mathcal{C}^{\mathrm{Ch}}_I(x)\mathfrak{t}'_I.
\]

Observe that if $I$ consists of $i\ge 1$ integers $p-1$, then $\mathcal{C}^{\mathrm{Ch}}_I$ is the usual $i$-th Steenrod operation $\mathrm{St}^i$ mod $p$.
\begin{ex}
For the projective space $X=\mathbb{P}^n$ of Example~\ref{ex:LNP} we have $\mathrm{Ch}(X)=\mathbb{F}_p[c]/(c^{n+1})$ and $\mathcal{C}^{\mathrm{Ch}}(c)=c+t_{p-1}c^p+t_{p^2-1}c^{p^2}+\ldots$.
\end{ex}

There is a commutative diagram of morphisms of FGLs 
\[
\xymatrix{
F_U \ar[r]^{(\phi,\lambda_{\mathfrak{t}})}\ar[d]_{(\phi_0,\mathrm{id})} & F_U[\mathfrak{t}] \ar[d]^{\mathfrak{t} \to \mathfrak{t}'}\\
F \ar[r]^{(\mathrm{id},\lambda_{\mathfrak{t}'})}& F[\mathfrak{t'}],
}
\]
where $\phi(a_{ij})=\tilde{a}_{ij}$, $\phi_0(a_{ij})=0$, which gives rise to a commutative diagram of the respective operations 
\[
\xymatrix{
\Omega^*(X) \ar[r]^{\mathcal{C}_I^\Omega}\ar[d] & \Omega^{*+|I|}(X) \ar[d]\\
 \mathrm{Ch}^*(X) \ar[r]^{\mathcal{C}_I^{\mathrm{Ch}}}& \mathrm{Ch}^{*+|I|}(X)
}.
\]

\begin{rem}
The Steenrod operations can be also constructed as follows (see \cite[\S6.4]{Vi17}):

Consider a multiplicative operation $\mathrm{St}\colon \mathrm{Ch}\to \mathrm{Ch}[t]$, $\deg t=1$, corresponding to the morphism of FGLs $(\mathrm{id},-t^{p-1}x +x^p)\colon F\to F[t]$, where $F[t]$ is the additive FGL  over $\mathbb{F}_p[t]$. Observe that $\mathrm{St}$ does not preserve the grading. Then $\mathrm{St}^i\colon \mathrm{Ch}^m \to \mathrm{Ch}^{m+i(p-1)}$ is given by taking the coefficient at $t^{(m-i)(p-1)}$ of the image of $\mathrm{St}$.
\end{rem}

%%%%%%%%%%%%%%%%%%%%%%%%%%%%%
%%%%%%%%%%%%%%%%%%%%%%%%%%%%%
%%%%%%%%%%%%%%%%%%%%%%%%%%%%%

\section{Localized operations on $T$-equivariant theories}\label{sec:main}

\subsection{\it Equivariant coefficient rings}
Let $\Lambda$ be a free abelian group of finite rank and let $F$ be a FGL over a ring $R$. 
Following \cite[\S2]{CPZ} consider the polynomial ring $R[x_\la]_{\la\in \Lambda}$ in variables corresponding to elements of $\Lambda$. Let $R[[x_\lambda]]_{\lambda\in \Lambda}$ be its completion with respect to the kernel of the augmentation map $x_\la\mapsto 0$. Set
\[
S_F(\Lambda):=R[[x_\la]]_{\la\in \Lambda}/\overline{J},
\]
where $\overline{J}$ is the closure of the ideal of relations $x_0=0$ and $x_{\la+\mu}=x_\la+_F x_\mu$, $\la,\mu\in \Lambda$. The ring $S_F(\Lambda)$ is called the formal group ring associated to the lattice $\Lambda$ and to the FGL $F$ over $R$. 

\begin{ex} 
(i) For the additive FGL $S_F(\Lambda)=\mathrm{Sym}_R(\Lambda)^\wedge$ is the usual completed symmetric algebra of $\Lambda$ over $R$ \cite[Example~2.19]{CPZ}.

(ii) For the multiplicative periodic FGL, $S_F(\Lambda)=R[\Lambda]^\wedge$ is the completed group ring of $\Lambda$ over $R$ \cite[Example~2.20]{CPZ}.
\end{ex}

The following Lemma~\ref{lem:fgr} provides an analogue of the map~\eqref{eq:vfgl} for formal group rings
\begin{lem}\label{lem:fgr} 
Let $(\phi,\gamma)\colon F_1/R_1 \to F_2/R_2$ be a morphism of FGLs. 

Then there is an induced ring homomorphism $\mathcal{C}\colon S_{F_1}(\Lambda) \to S_{F_2}(\Lambda)$ such that $\mathcal{C}(x_\la)\in x_\la S_{F_2}(\Lambda)$, for all $\la\in\Lambda$. 
\end{lem}
We call the map $\mathcal{C}$ the operation on formal group rings corresponding to $(\phi,\gamma)$.

\begin{proof}
By \cite[\S2]{CPZ} the formal group ring is functorial with respect to all the parameters: coefficients,  lattices and FGLs. Set $\mathcal{C}=\gamma^\star\circ \phi_\star$ to be the composite of functorial maps of \cite[Lemma~2.6]{CPZ} induced by $\phi \colon R_1\to R_2$ and $\gamma\colon F_2 \to \phi(F_1)$. Since $\phi_\star\colon S_{F_1}(\Lambda) \to S_{\phi(F_1)}(\Lambda)$, $x_\lambda\mapsto x_\lambda$ and $\gamma^\star\colon S_{\phi(F_1)}(\Lambda) \to S_{F_2}(\Lambda)$, $x_\lambda \mapsto \gamma(x_\lambda)$, we obtain $\mathcal{C}(x_\la)\in x_\la S_{F_2}(\Lambda)$. 
\end{proof}

Following~\cite[\S5]{HML} given a free AOCT $\hh$ with the FGL $F/R$, a split torus $T$ and a smooth $T$-variety $X$ over $k$, one constructs the associated equivariant oriented cohomology $\hh_T(X)$ as the limit of usual oriented cohomology taken over a certain system of $T$-representations (the Borel construction). We refer to \cite{To,EG,Br97} for the definition and properties of the equivariant Chow theory and $K$-theory; and to \cite{Kr12} for the $T$-equivariant algebraic cobordism.
All such $T$-equivariant theories satisfy the axioms of \cite[\S2]{CZZ2}. 
For simplicity of the exposition we assume
that our theories are Chern-complete, i.e. $\hh_T(pt)=\hh_T(pt)^\wedge$.

According to \cite[\S3]{CZZ2} the formal group ring of the character group $\Lambda=T^*$ can be identified with the  $T$-equivariant coefficient ring $\hh_T(pt)$, i.e. 
\[
\hh_T(pt) \simeq S_F(T^*).
\]
Under this identification $x_\la$ in the definition of the formal group ring corresponds to the first characteristic class $c_1^{\hh}(\Ll_\la)$ of the associated $T$-equivariant line bundle $\Ll_\la$ over $pt$. 

\subsection{\it The fixed loci}
Let $G$ be a split semi-simple linear algebraic group $G$ over a field $k$ of characteristic $0$ together with a a split maximal torus $T$ and a Borel subgroup $B$ containing $T$. Consider the associated finite crystallographic (semisimple) root datum $\Sigma$ together with its subset of simple roots $\Delta$ and the character lattice $\Lambda=T^*$. The Weyl group $W$ of $\Sigma$ acts on $T^*$ and, hence, by functoriality on the formal group ring $S_F(T^*)$. This action coincides with the natural action of $W$ on $\hh_T(pt)$.

Set $S=S_F(T^*)$.
We define the twisted group algebra $S_W$ to be the free left $S$-module $S\otimes_R R[W]$ so that each element of $S_W$ can be written uniquely as a linear combination $\sum_{w\in W} p_w\de_w$, where $\{\de_w\}_{w\in W}$ is the standard basis of the group ring $R[W]$ and $p_w\in S$. 
It has the product induced  by that of $R[W]$ and the twisted commuting relation $w(p)\de_w =\de_w p$,  $p\in S$.

Consider its $S$-linear dual $S_W^\star=\Hom_S (S_W,S)$. By definition $S_W^\star$ is a free $S$ module with basis $\{f_w\}_{w\in W}$ given by $f_w(v)=\delta_{w,v}$. 
It can be identified with
\[
\Hom_S (S_W,S)= \Hom(W,S) = \bigoplus_{w\in W} S,\qquad \sum_w p_wf_w \mapsto (p_w)_{w\in W},
\]
so that $S_W^\star$ is a commutative $S$-algebra where the product is given by the usual point-wise multiplication of functions $W\to S$. In particular, the unit in $S_W^\star$ is given by $\sum_w f_w$.
 
Now let $\Theta\subset \Delta$ and $W_\Theta$ be the subgroup of $W$ generated by simple reflections corresponding to $\Theta$. Similar to $S_W$ we define the parabolic $S$-module $S_{W/W_\Theta}$ (see \cite[\S11]{CZZ1}) so that its dual $S_{W/W_\Theta}^\star$ is a commutative $S$-algebra with a basis $\{f_{\bar{w}}\}_{\bar w}\in W/W_\Theta$ indexed by cosets $W/W_\Theta$. 

\begin{dfn}
Let $(\phi,\gamma)\colon F_1/R_1 \to F_2/R_2$ be a morphism of FGLs. Let $S_i=S_{F_i}(T^*)$, $i=1,2$ denote the respective formal group ring. Taking the direct sum over $W/W_\Theta$ of the operations of Lemma~\ref{lem:fgr} defines a ring homomorphism 
\[
\mathcal{C}_{W/W_\Theta}\colon  (S_{1})_{W/W_\Theta}^\star \to (S_2)_{W/W_\Theta}^\star,\quad \mathcal{C}_{W/W_\Theta}(\sum_{\bar{w}} p_{\bar{w}} f_{\bar{w}})=\sum_{\bar{w}} \mathcal{C}(p_{\bar{w}})f_{\bar{w}}
\] 
which we call the operation on fixed loci induced by $(\phi,\gamma)$.
\end{dfn}

\subsection{\it Equivariant cohomology}
Following  \cite{CZZ,CZZ1,CZZ2} we recall computations (which generalize computations by Kostant-Kumar \cite{KK86, KK90}) of the $T$-equivariant AOCT of the variety of complete flags $\hh_T(G/B)$. For these computations to work we need to assume that 
\begin{quote}
The formal group ring $S_F(T^*)$ is $\Sigma$-regular of \cite[Definition~4.4]{CZZ}.
\end{quote} 
In particular, this holds if $2$ is a regular element in $R$ or if the root system does not contain any symplectic root subsystem as an irreducible component.

As the first step consider the localizations at all $x_\al$'s, $\alpha\in \Sigma$
\[
Q=S[\tfrac{1}{x_\al},\al\in \Sigma]\quad\text{ and }\quad Q_W=Q \otimes_R R[W].
\]
Following \cite[\S6]{HMSZ} we define the formal affine Demazure algebra $\DF$ to be the subalgebra of the localized twisted group algebra $Q_W$ generated by elements of $S\subset Q_W$ and by the push-pull elements 
\[
Y_\alpha=\tfrac{1}{x_{-\al}}+\tfrac{1}{x_{\al}}\de_{s_{\alpha}}\in Q,
\] 
for all  roots $\al$ and the corresponding  reflections $s_\alpha\in W$. 

\begin{ex} (cf. \cite[\S7]{HMSZ}) 
For the additive FGL, $\DF$ is the (completed) nil-Hecke algebra. For the multiplicative FGL, $\DF$ is the (completed) $0$-affine Hecke algebra.
\end{ex}

Now by \cite[Theorem~10.7]{CZZ1} the natural inclusion $S_W \hookrightarrow \DF$ induces a ring inclusion $\DFd \hookrightarrow S_W^\star$ on the $S$-linear duals (called a restriction to fixed locus). By \cite[Theorem~8.2]{CZZ2} the $T$-equivariant AOCT of $G/B$ can be identified with $\DFd$ in such a way that the restriction to  fixed locus coincides with the pull-back $\hh_T(G/B) \hookrightarrow \hh_T((G/B)^T)$.

Finally, consider the projective homogeneous space $X=G/P_\Theta$, where $\Theta\subset \Delta$ and $P_\Theta$ is the standard parabolic subgroup containing $B$. 
There is the parabolic analogue of all the above results (see \cite{CZZ1} and \cite{CZZ2}) which says that there is the ring inclusion
\[
\hh_T(G/P_\Theta)\simeq \mathbb{D}_{F,\Theta}^\star \hookrightarrow S_{W/W_\Theta}^\star \simeq \hh_T((G/P_\Theta)^T), 
\]
where $\mathbb{D}_{F,\Theta}$ is the respective parabolic module defined in \cite[\S11]{CZZ1}.

\subsection{\it The forgetful map} 
We now recall the construction of the forgetful map (see e.g. \cite[\S11]{CZZ}). We assume for simplicity $\Theta=\emptyset$.

Let $\mathcal{I}$ denote the kernel of the augmentation map $\epsilon\colon S\to R$. Consider the quotient map (here $\DFd$ is an $S$-module)
\[
\varepsilon\colon \DFd \to \DFd/\mathcal{I}\DFd.
\]
Given $w\in W$ and its reduced expression 
\[
w=s_{\alpha_{i_1}}s_{\alpha_{i_2}}\ldots s_{\alpha_{i_l}},\quad \alpha_{i_j}\in \Delta,
\] 
we set $I_w=(i_1,i_2,\ldots,i_l)$ and $Y_{I_w}=Y_{\alpha_{i_1}}\ldots Y_{\alpha_{i_l}}$ in $Q_W$. By \cite[Prop.~7.7]{CZZ} the set $\{Y_{I_w}\}_{w}$ forms an $S$-basis of the algebra $\DF$. By \cite[\S10]{CZZ1} the $Q$-dual set $\{Y_{I_w}^*\}_w$ forms an $S$-basis of the dual $\DFd$. Finally, the set $\{\varepsilon(Y_{I_w}^*)\}_w$ forms an $R$-basis of the quotient which can be identified with the usual AOCT $\hh(G/B)$ (see e.g. \cite[\S11]{CZZ}).

\begin{rem}
Instead of the basis $\{Y_{I_w}^*\}_w$ we can take any of the $S$-bases of $\DFd$ discussed in \cite[\S15]{CZZ1}. In each case the map $f$ sends the respective coefficients $p_w\in S$ to their `constant terms' $\epsilon(p_w)\in R$. For instance, we can take the Schubert basis $\{\zeta_{I_w}=Y_{I_w}([pt])\}_w$ which is Poincar\'e dual to $\{Y_{I_w}^*\}_w$.
\end{rem}

In view of the above identifications the map $\varepsilon$ coincides with the usual forgetful map
\[
\hh_T(G/B)\simeq \DFd\longrightarrow \DFd/\mathcal{I}\DFd \simeq \hh(G/B).
\]
As before, this description can be extended to the parabolic situation in which case we obtain the forgetful map
\[
\varepsilon\colon \hh_T(G/P_\Theta)\simeq \mathbb{D}_{F,\Theta}^\star  \longrightarrow  \mathbb{D}_{F,\Theta}^\star/\mathcal{I} \mathbb{D}_{F,\Theta}^\star \simeq \hh(G/P_\Theta).
\]

\begin{rem} 
The fact that the quotient $\mathbb{D}_{F,\Theta}^\star/\mathcal{I}\mathbb{D}_{F,\Theta}^\star$ for the additive FGL coincides with the usual Chow ring $CH(G/P_\Theta)$ and the map $\varepsilon$ coincides with the forgetful map $CH_T(G/P_\Theta) \to CH(G/P_\Theta)$ can be also found in \cite[(5.12)]{KK86}.
\end{rem}

\subsection{\it The characteristic map}
The Weyl group $W$ acts by $\Z$-linear automorphisms of $\Lambda$ and any such automorphism extends to an automorphism of the formal group ring $S=S_F(\Lambda)$. Let $\Theta\subset \Delta$ and $W_\Theta$ be the subgroup of $W$ generated by simple reflections corresponding to $\Theta$. Denote by $S^{W_\Theta}$ the subring of invariants. Consider the characteristic map 
\[
c\colon S^{W_\Theta} \longrightarrow S_{W/W_\Theta}^\star,\quad  p\mapsto (w(p))_{w}
\] 
and the Borel map 
\[
\rho \colon S\otimes_{S^W} S^{W_\Theta} \longrightarrow S_{W/W_\Theta}^\star,\quad  p_1\otimes p_2\mapsto p_1c(p_2).
\] 
Since the functorial maps on formal group rings induced by $W$-automorphisms commute with the maps induced by morphisms of formal group laws, we obtain

\begin{lem}\label{lem:char} 
Operations on fixed loci commute with the characteristic map and the Borel map, i.e. $c\circ \mathcal{C}=\mathcal{C}_{W/W_\Theta} \circ c$ and $\rho\circ (\mathcal{C}\otimes \mathcal{C})=\mathcal{C}_{W/W_\Theta} \circ \rho$.
\end{lem}

We are now ready to state and to prove our key result
\begin{thm}\label{thm:main}
Let $(\phi,\gamma)\colon F_1/R_1 \to F_2/R_2$ be a morphism of FGLs. Let $G$ be a split semisimple linear algebraic group over a field of characteristic $0$ together with a maximal split torus $T$ and a Borel subgroup $B$ containing $T$. Let $\Sigma$ be the respective root system together with a subset $\Theta$ of the set of simple roots $\Delta$. Let $P_\Theta$ be the parabolic subgroup for $\Theta$ and $G/P_\Theta$ the respective projective homogeneous variety.  Let $S_1$ and $S_2$ be the formal group rings associated to $F_1$, $F_2$ and the character lattice $\Lambda=T^*$. Assume that both $S_1$ and $S_2$ are $\Sigma$-regular.

(i) The operation on fixed loci  
$\mathcal{C}_{W/W_\Theta}\colon (S_1)_{W/W_\Theta}^\star \to (S_2)_{W/W_\Theta}^\star$ restricts to a ring homomorphism between the respective $T$-equivariant oriented cohomology theories 
\[
\mathcal{C}_T\colon (\hh_1)_T(G/P_\Theta) \to (\hh_2)_T(G/P_\Theta),
\]
where $\hh_1$ is a free theory for the FGL $F_1$ and $\hh_2$ is a theory with the FGL $F_2$. Moreover, $\mathcal{C}_T$ commutes with the characteristic and the Borel map.

(ii) There is a commutative diagram
\[
\xymatrix{
(\hh_1)_T(G/P_\Theta) \ar[d]_{\mathcal{C}_T} \ar[r]^\varepsilon & \hh_1(G/P_\Theta)  \ar[d]^{\mathcal{C}} \\
(\hh_2)_T(G/P_\Theta) \ar[r]^\varepsilon  & \hh_2(G/P_\Theta)  
}
\]
where $\varepsilon$ is the forgetful map and $\mathcal{C}$ is the multiplicative operation induced by $(\phi,\gamma)$.
\end{thm}

\begin{dfn}
We call the map $\mathcal{C}_T$ a localized multiplicative operation on $T$-equivariant AOCTs induced by $(\phi,\gamma)$.
\end{dfn}

\begin{proof} (i) Assume $\Theta=\emptyset$, i.e. $G/P_\Theta=G/B$.
By \cite[Thm.~10.7]{CZZ1} the image of the restriction to fixed locus $\DFd \hookrightarrow S_W^\star$ and, hence, the equivariant cohomology can be described as
\[
\hh_T(G/B)\simeq \{\sum_w p_w f_w\in S_W^\star \mid p_w-p_{s_\alpha w} \in x_{\alpha}S,\; \alpha\in \Sigma^+,\; w\in W \}.
\]
By Lemma~\ref{lem:fgr} we have
\[
\mathcal{C}(p_w)-\mathcal{C}(p_{s_\alpha w})\in \mathcal{C}(x_\alpha S_1)\subset x_\alpha S_2.
\]
So the operation $\mathcal{C}_W$ preserves the above relations and, therefore, it restricts to the equivariant cohomology. We denote the restriction by $\mathcal{C}_T$. By Lemma~\ref{lem:char} it commutes with the characteristic and Borel maps.

The case of a general $\Theta$ follows by the same arguments using the parabolic analogue \cite[Thm.~11.9]{CZZ1} of \cite[Thm.~10.7]{CZZ1}.

(ii) Since $\mathcal{C}$ preserves the augmentation ideal, the operation $\mathcal{C}_T$ induces the map on quotients 
\[
\tilde{\mathcal{C}}\colon \hh_1(G/P_\Theta)\simeq \mathbb{D}_{F_1,\Theta}^\star/\mathcal{I}_1\mathbb{D}_{F_1,\Theta}^\star \longrightarrow \mathbb{D}_{F_2,\Theta}^\star/\mathcal{I}_2\mathbb{D}_{F_2,\Theta}^\star \simeq \hh_2(G/P_\Theta)
\]
which also commutes with the characteristic map. 

Observe that the multiplicative operation $\mathcal{C}$ corresponding to the morphism $(\phi,\gamma)\colon F_1/R_1 \to F_2/R_2$ maps $c_1^{\hh_1}(\Ll)$ to $\gamma(c_1^{\hh_2}(\Ll))$ for any line bundle $\Ll$ over $G/B$ (e.g. see \cite[\S2.5]{Pa04}) and the same holds for the operation $\tilde{\mathcal{C}}$ since it commutes with the characteristic map. 

Consider the canonical map $g\colon \Omega \to \hh_1$. It maps the $\LL$-basis $\{\varepsilon( Y_{I_w}^*)\}_w$ of $\Omega(G/B)$ to the respective $R_1$-basis of $\hh_1(G/B)$ and coefficients in $\LL$ to the coefficients of $F_1$ in $R_1$. It also maps $c_1^\Omega(\Ll) \mapsto c_1^{\hh_1}(\Ll)$. Both operations $\tilde{\mathcal{C}}$ and $\mathcal{C}$ can be uniquely extended to the operations $\Omega(G/B) \to \hh_2(G/B)$ via $g$. Moreover, these extensions  will coincide on characteristic classes of line bundles. Since the characteristic classes $c_1^{\Omega}(\Ll)$ generate $\Omega(G/B)$ rationally by \cite[Corollary~13.9]{CPZ}, the 
operations $\tilde{\mathcal{C}}$ and $\mathcal{C}$ coincide over $G/B$.

Finally, observe that the pull-back $\pi^*\colon \hh(G/P_\Theta) \to \hh(G/B)$ is the inclusion and both operations $\mathcal{C}$ and $\tilde{\mathcal{C}}$ commute with $\pi^*$ by definition.
\end{proof}

\subsection{\it Localized LN operations and their Chow traces}
By Theorem~\ref{thm:main} we can extend the LN operations and their Chow traces as follows:

Recall that both these operations are projections of  the graded multiplicative operations $\mathcal{C}^\Omega\colon \Omega \to \Omega[\mathfrak{t}]$ and $\mathcal{C}^{\mathrm{Ch}}\colon \mathrm{Ch} \to \mathrm{Ch}[\mathfrak{t}']$ induced by the morphisms of FGLs $(\phi,\lambda_{\mathfrak{t}})\colon F_U \to F_U[\mathfrak{t}]$ and $(\mathrm{id},\lambda_{\mathfrak{t}'})\colon F_a\to F_a[\mathfrak{t}']$ respectively (here $F_a$ denotes the additive FGL).
We call the associated localized multiplicative operations
\[
\mathcal{C}_T^\Omega\colon \Omega_T(G/P_\Theta) \to \Omega[\mathfrak{t}]_T(G/P_\Theta)\quad\text{ and }\quad
\mathcal{C}_T^{\mathrm{Ch}}\colon \mathrm{Ch}_{T}(G/P_\Theta) \to \mathrm{Ch}[\mathfrak{t}']_T(G/P_\Theta)
\]
the localized total LN operation and the localized total Chow trace operation, respectively.

Finally, let $\psi_I$ be the $\LL$-linear projection $\LL[\mathfrak{t}] \to \LL$ to the respective monomial $\mathfrak{t}_I$. We call the composite of the total operation $\mathcal{C}_T^\Omega$ with the induced projection 
\[
-\otimes\psi_I\colon \Omega[\mathfrak{t}]_T(G/P_\Theta)=\Omega_T(G/P_\Theta)\otimes_{\LL}{\LL[\mathfrak{t}]} \longrightarrow \Omega_T(G/P_\Theta)
\] 
by the localized LN operation corresponding to the partition $I$. 

Similarly, we call the composite of $\mathcal{C}_T^{\mathrm{Ch}}$ with the induced projection 
\[
-\otimes\psi'_{I}\colon \mathrm{Ch}[\mathfrak{t}']_T(G/P_\Theta)=\mathrm{Ch}_T(G/P_\Theta)\otimes_{\mathbb{F}_p}{\mathbb{F}_p[\mathfrak{t}']} \to \mathrm{Ch}_T(G/P_\Theta)
\] 
by the localized Chow trace  corresponding to the partition $I$ (here $\psi'_I\colon \mathbb{F}_p[\mathfrak{t}'] \to \mathbb{F}_p$ is the projection corresponding to $\mathfrak{t}'_I$).

%%%%%%%%%%%%%%%%%%%%%%%%%%%%%
%%%%%%%%%%%%%%%%%%%%%%%%%%%%%
%%%%%%%%%%%%%%%%%%%%%%%%%%%%%

\section{Computation of localized operations}\label{sec:compop}

In the present section we describe an algorithm that can be used to compute multiplicative operations on AOCTs of projective homogeneous varieties. This extends the algorithm provided in \cite{GPS} for the Steenrod operations on Chow groups.

Our setup is as follows. Suppose we are given a multiplicative operation $\hh_1\to \hh_2$ corresponding to the morphism $(\phi,\gamma)\colon F_1/R_1 \to F_2/R_2$ of FGLs. Consider the respective (localized) multiplicative operations $\mathcal{C}_T\colon (\hh_1)_T(G/B) \to (\hh_2)_T(G/B)$ and  $\mathcal{C}\colon \hh_1(G/B) \to \hh_2(G/B)$. Set $S_1=(\hh_1)_T(pt)$ and $S_2=(\hh_2)_T(pt)$.

We choose an $S_1$-basis $\mathcal{B}_1$ of $(\hh_1)_T(G/B)$ and an $S_2$-basis $\mathcal{B}_2$ of $(\hh_2)_T(G/B)$. We want to compute the action of $\mathcal{C}_T$ on elements of $\mathcal{B}_1$ and express it in terms of $\mathcal{B}_2$, i.e. for each $f\in \mathcal{B}_1$ we want to compute the coefficients $c_{f,g} \in S_2$ such that
\[
\mathcal{C}_T(f)=\sum_{g\in \mathcal{B}_2} c_{f,g} g. 
\]

Applying the forgetful map $\varepsilon$ we then restrict it to the respective $R_1$-basis of $\hh_1(G/B)$ and the $R_2$-basis of $\hh_2(G/B)$ and, hence, compute the usual multiplicative operation
\[
\mathcal{C}(\varepsilon(f))=\sum_{g \in \mathcal{B}_2} \epsilon(c_{f,g}) \varepsilon(g).
\]

Below we focus on the cases when both $\mathcal{B}_1$ and $\mathcal{B}_2$ are either the $Q$-dual bases $\{Y_{I_w}^*\}_w$ or the Schubert bases $\{\zeta_{I_w}\}_w$. Note that our algorithms can be easily extended to a general parabolic case, i.e. operations on $G/P_\Theta$ using the results of \cite{CZZ1}.

\subsection{\it The $Q$-dual basis} 
Consider the basis of $\hh_T(G/B)$ given by the $Q$-dual elements $\{Y_{I_w}^*\}_w$. Dualizing the presentation $\delta_v=\sum_{w\le v} b_{v,I_w}Y_{I_w}$, $b_{v,I_w}\in S=\hh_T(pt)$ we obtain (see \cite[(7.1)]{CZZ1})
\[
Y_{I_w}^*= \sum_v b_{v,I_w} f_v \in \hh_T^{l(w)}(G/B),\;\text{where}
\]
\begin{itemize}
\item[(i)] $b_{v,I_w}=0$ for all $v\not\ge w$, and
\item[(ii)] $b_{w,I_w}=\prod_{\alpha\in \Sigma^+\cap w\Sigma^-} x_\alpha$ (see \cite[Lemma~3.3]{CZZ1}).
\end{itemize}
Hence, the transformation matrix $C=(b_{v,I_w})_{v,w}$ over $S$ whose columns are the basis vectors $Y_{I_w}^*$ is low-triangular with monomials $b_{w,I_w}\in S^{(l(w))}$ on the main diagonal.

Consider the inverse matrix $C^{-1}=(a_{I_v,w})_{v,w}$ over $Q$, where $Y_{I_v}=\sum_{w\le v} a_{I_v,w}\delta_w$ with $a_{I_w,w}=1/b_{w,I_w}$. By definition $C^{-1}\cdot (p_v)_v$ expresses $x=\sum_v p_v f_v$, $p_v\in S$ in terms of the basis $\{Y_{I_w}^*\}_w$. Since $C^{-1}$ is also low-triangular, we obtain
\begin{itemize}
\item[(iii)] If $x=\sum_v p_vf_v$ is such that $p_v=0$ for all $v\not\ge w$, then $a_{I_w,w}\cdot p_w \in S$.
\end{itemize}

\begin{rem} 
The algorithm for the Chow theory in \cite{GPS} uses the notation $Z_w^T$ for $Y_{I_w}^*$, $x=(p_v)_v$, $\alpha=x_\alpha$, $\imath_u(x)=p_v$, $\imath_v(Z_w^T)=b_{v,I_w}$. The statements (1-3) of~\cite[Lemma~5.5]{GPS} are precisely the properties (i-iii) above in the case of the additive FGL.
\end{rem}

\begin{rem}
The coefficients $b_{v,I_w}$ and, hence, the matrix $C$ can be efficiently computed using the root polynomial method of \cite{Bi99}, \cite{Wi04} for the Chow theory and the K-theory, respectively. The generalization of this approach to the case of  a hyperbolic FGL can be found in \cite[Theorem~4.4]{LZ17}. 
\end{rem}

\subsection{\it The elimination procedure}
We now extend the elimination procedure of \cite[5.6]{GPS} to an arbitrary FGL and the associated structure algebra. Roughly speaking, this procedure computes the product $(a_w)_w=C^{-1}\cdot (p_v)_v$ using the divisibility property~(iii). We follow closely \cite{GPS}:

We first extend the Bruhat order to a linear order on the Coxeter group $W$. Given a non-zero $x=\sum_v p_vf_v \in \hh_T^{m}(G/B)$ let $u\in W$ be the minimal element such that  $p_u\neq 0$. Then by~(iii) $p_u$ is divisible by the monomial $b_{u,I_u}$. Since $b_{u,I_u} \in S^{(l(u))}$, we get $l(u)\le m$.

Assume $l(u)<m$. Then we compute the quotient $b_u:=p_u/b_{u,I_u} \in S^{(m-l(u))}$ and set $x'=x-b_u Y_{I_u}^*$. Repeat the procedure to $x'=\sum_v p_v' f_v$ and so on. Observe that by construction $p_u'=0$ and $p_w'=0$ for all $w<u$. Hence, this procedure will stop when either $x=0$ or $l(u)=m$.

Suppose $l(u)=m$. Consider all $v\in W$ such that $l(v)=m$. For all such $v$'s $b_v:=p_v/b_{v,I_v} \in S^{(0)}$. Set $y:=x-\sum_{l(v)=m} b_v Y_{I_v}^*$.

Assume $y=\sum_v y_vf_v \neq 0$. Let $u$ be the minimal element such that $y_u\neq 0$, so $l(u)>m$. Then $y_u/b_{u,I_u} \in S$ so $m\ge l(u)$, a contradiction. Hence, $y=0$.

\subsection{\it Algorithm for the $Q$-dual basis}
The action of $\mathcal{C}_T$ on the basis $\{Y_{I_w}^*\}_w$ of $(\hh_1)_T(G/B)$ can be computed as follows:

\subparagraph{I} 
First, one computes the matrix $C=(b_{v,I_w})_{v,w}$, hence, expressing each $Y_{I_w}^*$ as an element $\sum_v b_{v,I_w}f_v$ of $(\hh_1)_T(G/B)$.

\subparagraph{II} 
Then one finds $p_v=\mathcal{C}(b_{v,I_w})$ in $S_2$ for each of the coordinates and obtains the element $x=\sum_v p_v f_v$ of 
$(\hh_2)_T(G/B)$.

\subparagraph{III} 
Finally, one applies the elimination procedure to $x$ and computes $(a_w)_w=C^{-1}\cdot (p_v)_v$, hence, expressing $x$ in terms of the basis $\{Y_{I_w}^*\}_w$ in $(\hh_2)_T(G/B)$. 

\subsection{\it Algorithm for the Schubert basis} 
Set $x_\Pi=\prod_{\alpha \in \Sigma^-}x_\alpha \in S=\hh_T(pt)$. By the (virtual) class of a point denoted $[pt]$ we call the element $x_\Pi f_1 \in \hh_T(G/B)$. Define
\begin{equation}\label{eq:for2}
\zeta_{I_w}=Y_{I_w^{-1}}([pt])=Y_{\alpha_{i_l}}(Y_{\alpha_{i_{l-1}}}(\ldots Y_{\alpha_1}([pt]))).
\end{equation}
The family $\{\zeta_{I_w}\}_w$ forms an $S$-basis of $\hh_T(G/B)$ called the Schubert basis. It has the following properties:

(i)
By \cite[Theorem~12.4]{CZZ1} it is Poincar\'e dual to the $Q$-dual basis $\{Y_{I_{w}}^*\}_w$, namely
\[
A_\Pi(Y_{I_w}^* \cdot \zeta_{I_v}) =\delta_{w,v}\mathbf{1}, \quad\text{ where } \mathbf{1}=\sum_u f_u,\quad A_\Pi (\sum_v p_vf_v)=(\sum_{v} \tfrac{p_v}{v(x_\Pi)})\cdot\mathbf{1}.
\]

(ii)
If $\zeta_{I_w}=\sum_v p_v^{I_w}f_v$, then $p_v^{I_w}=0$ for all $v\not\le w$. 

(iii)
$\zeta_{I_w}\in \hh_T^{N-l(w)}(G/B)$,  where $N=\dim G/B$.

(iv)
By \cite[Lemma~7.3]{CZZ1} we have $p_v^{I_w}=v(x_\Pi)a_{I_w,v}$. In particular, we have
\[
p_w^w=p_w^{I_w}=w(x_\Pi)/b_{w,I_w}.
\]

All these conditions imply that the transformation matrix $D:=(p_v^{I_w})_{w,v}$ over $S$ (the rows of $D$ are the Schubert classes) is low-triangular with the elements $p_w^w$ on the main diagonal.

In the opposite direction, let $D^{-1}=(\tfrac{1}{v(x_\Pi)}b_{v,I_w})_{v,w}$ be the inverse matrix over $Q$ (the rows of $D^{-1}$ are the Schubert coordinates of $(\delta_{u,v})_u$ in $Q$). Then the Schubert coordinates of an element $\sum_v p_v f_v \in \hh_T(G/B)$ are given by $(p_v)_v  \cdot D^{-1}$.

The action of $\mathcal{C}_T$ on the Schubert basis $\{\zeta_{I_w}\}_w$ of $(\hh_1)_T(G/B)$ can be computed as follows:

\subparagraph{I} 
First one computes the matrix $D$ expressing $\zeta_{I_w}$ as $\sum_v p_v^{I_w} f_v$ of $(\hh_1)_T(G/B)$. For this one uses the inductive formulas \eqref{eq:for2}.

\subparagraph{II} 
Then one computes $p_v=\mathcal{C}(p_v^{I_w})$ in $S_2$ for each of the coordinates and obtains the element $x=\sum_v p_v f_v$ of 
$(\hh_2)_T(G/B)$.

\subparagraph{III} 
Finally, one finds the matrix $D^{-1}$ using the formulas for the coefficients $b_{v,I_w}$ and computes $(p_v)_v \cdot D^{-1}$, hence, expressing $x$  in terms of the Schubert basis $\{\zeta_{I_w}\}_w$ in $(\hh_2)_T(G/B)$.

%%%%%%%%%%%%%%%%%%%%%%%%%%%%%
%%%%%%%%%%%%%%%%%%%%%%%%%%%%%
%%%%%%%%%%%%%%%%%%%%%%%%%%%%%

\section{Cohomological operations and push-pull operators} 

In this section we study behaviour of cohomological operations on oriented cohomology of additive types with respect to the (Hecke) action by push-pull operators.

Following \cite[\S4]{CZZ1} we define the Hecke action of $Q_W$ on $Q_W^*=Hom_Q(Q_W,Q)$ as follows:
\[
(z\bullet f)(z'):=f(z'z), \quad z,z'\in Q_W,\; f\in Q_W^*.
\]
By definition, this action is left $Q$-linear, i.e. $z\bullet (qf)=q(z\bullet f)$ and it induces  a $Q$-linear action of $W$ on $Q_W^*$ via $w(f):=\delta_w\bullet f$. We have $q\bullet f_w=w(q)f_w$ and $w(f_v)=f_{vw^{-1}}$ for any $q\in Q$ and $w,v\in W$. By \cite[Theorem~10.13]{CZZ1} this action restricts to the action of $\DF$ on $\DFd\simeq \hh_T(G/B)$ so that $\hh_T(G/B)$ becomes a free $\DF$-module of rank one.

Consider the push-pull element $Y_\alpha=\tfrac{1}{x_{-\alpha}}+\tfrac{1}{x_\alpha}\delta_{s_\alpha}\in \DF$. It acts on $\DFd$ as follows
\begin{align*}
Y_\alpha \bullet (\sum_v p_v f_v) &=\sum_v \tfrac{p_v}{v(x_{-\alpha})}f_v+\sum_v \tfrac{p_v}{vs_\alpha(x_\alpha)}f_{vs_\alpha}\\
 &=\sum_v (\tfrac{p_v}{x_{-v(\alpha)}}+\tfrac{p_{vs_\alpha}}{x_{v(\alpha)}})f_v\\
 &=\sum_v (\kappa_{v(\alpha)}p_v+\Delta_{v(\alpha)}(p_v))f_v,
\end{align*}
where $p_v\in S$, $\kappa_{v(\alpha)}=\tfrac{1}{x_{-v(\alpha)}}+\tfrac{1}{x_{v(\alpha)}}\in S$ and $\Delta_{v(\alpha)}(p_v)=(p_{s_{v(\alpha)}v}-p_v)/x_{v(\alpha)}\in S$ by \cite[Thm.~10.7]{CZZ1}.

\begin{dfn} 
We say that a FGL $F$ is of additive type if its formal inverse coincides with the usual additive inverse, i.e. $-_F x=-x$.
\end{dfn}
Observe that any $F$ of the form $F(x,y)=(x+y)g(x,y)$ for some $g\in R[[x,y]]$ is of additive type.

We are now ready to prove the following
\begin{prop} \label{thm:Riem}
Suppose $(\phi,\gamma)\colon F_1/R_1 \to F_2/R_2$ is a morphism of FGLs of additive type. Let $\mathcal{C}_T\colon (\hh_1)_T(G/B) \to (\hh_2)_T(G/B)$ be the localized multiplicative operation induced by $(\phi,\gamma)$. Let $Y_\alpha$ be the push-pull element acting on the respective oriented cohomology by means of the Hecke action.
 
Then we have
\[
\mathrm{itd}(x_{\alpha}) \bullet \mathcal{C}_T(Y_\alpha \bullet z) = Y_\alpha \bullet \mathcal{C}_T(z),\quad \text{for all }z\in (\hh_1)_T(G/B).
\]
\end{prop}

\begin{proof} 
Since $F_1$ and $F_2$ are of additive type, $\kappa_{v(\alpha)}=0$ for all $\alpha$ and $v$. 

Let $z=\sum_v p_v f_v \in (\hh_1)_T(G/B)$. Then we obtain
\begin{align*}
\mathrm{itd}(x_\alpha)\bullet \mathcal{C}_T(Y_\alpha \bullet z) &=\mathrm{itd}(x_\alpha)\bullet \mathcal{C}_T\big( \sum_v \Delta_{v(\alpha)}(p_v) f_v\big) \\
 &=\sum_v \tfrac{\gamma(x_{v(\alpha)})}{x_{v(\alpha)}}\mathcal{C}(\Delta_{v(\alpha)}(p_v))f_v =\sum_v \tfrac{\mathcal{C}(p_{s_{v(\alpha)}v})-\mathcal{C}(p_v)}{x_{v(\alpha)}}f_v.
\end{align*}

On the other side we have
\[
Y_\alpha\bullet \mathcal{C}_T(z) =Y_\alpha\bullet (\sum_v \mathcal{C}(p_v) f_v) =\sum_v \Delta_{v(\alpha)}(\mathcal{C}(p_v))f_v=\sum_v \tfrac{\mathcal{C}(p_{s_{v(\alpha)}v})-\mathcal{C}(p_v)}{x_{v(\alpha)}}f_v. \qedhere
\]
\end{proof}

\begin{rem}
Theorem~\ref{thm:Riem} obviously holds for the Chow traces of LN operations as well as for Steenrod operations.

In these cases it can be viewed as an extension of the classical formulas for the Steenrod operations (see e.g. \cite[Prop.~5.3]{Me03}) with respect to the push-forward $\mathrm{Ch}(G/B) \to \mathrm{Ch}(G/P_\Theta)$ induced by projection onto the flag variety corresponding to the minimal parabolic subgroup $P_\Theta$ with $\Theta=\{\alpha\}$.
\end{rem}

\bibliographystyle{alpha}

\end{document}